\documentclass[11pt]{amsart}
\usepackage{amssymb, latexsym}
\theoremstyle{plain}
\newtheorem{theorem}{Theorem}
\newtheorem{corollary}{Corollary}

\newtheorem{proposition}{Proposition}

\newtheorem*{1'}{Theorem 1-Bessel}
\newtheorem*{P2'}{Proposition 2-Bessel}
\newtheorem*{P3'}{Proposition 3-Bessel}
\newtheorem*{P4'}{Proposition 4-Bessel}
\newtheorem*{C1'}{Corollary 1-Bessel}

\newtheorem*{2'}{Theorem 2-Bessel}
\newtheorem*{3'}{Theorem 3-Bessel}

\theoremstyle{remark}

\newtheorem*{Remark 1}{Remark 1}
\newtheorem*{Remark 2}{Remark 2}
\newtheorem*{Remark 3}{Remark 3}
\newtheorem*{Remark 4}{Remark 4}

\numberwithin{equation}{section}

\begin{document}

\title [Clustering of  numbers in permutations avoiding a pattern]
{Clustering of consecutive numbers in permutations avoiding a pattern of length three or avoiding a finite number of  simple patterns}

\author{Ross G. Pinsky}

\address{Department of Mathematics\\
Technion---Israel Institute of Technology\\
Haifa, 32000\\ Israel}
\email{ pinsky@math.technion.ac.il}

\urladdr{https://pinsky.net.technion.ac.il/}

\date{}

\begin{abstract}
For $\eta\in S_3$,  let $S_n^{\text{av}(\eta)}$ denote the set of  permutations in $S_n$ that avoid the pattern $\eta$, and let $E_n^{\text{av}(\eta)}$ denote the expectation
with respect to the uniform probability measure on $S_n^{\text{av}(\eta)}$.
For $n\ge k\ge2$ and $\tau\in S_k^{\text{av}(\eta)}$,
let $N_n^{(k)}(\sigma)$ denote the number of occurrences  of $k$ consecutive numbers appearing in $k$ consecutive positions in $\sigma\in S_n^{\text{av}(\eta)}$,
and let $N_n^{(k;\tau)}(\sigma)$ denote the number of
  such occurrences for which the order of the  appearance of the $k$ numbers  is the pattern $\tau$.
  We obtain explicit
  formulas for
$E_n^{\text{av}(\eta)}N_n^{(k;\tau)}$ and $E_n^{\text{av}(\eta)}N_n^{(k)}$, for all $2\le k\le n$, all $\eta\in S_3$ and all $\tau\in S_k^{\text{av}(\eta)}$.
These exact formulas then yield  asymptotic formulas
as $n\to\infty$ with $k$ fixed,  and as $n\to\infty$ with  $k=k_n\to\infty$.
We also obtain analogous results for
 $S_n^{\text{av}(\eta_1,\cdots,\eta_r)}$,
 the subset of $S_n$ consisting of permutations avoiding the patterns $\{\eta_i\}_{i=1}^r$, where $\eta_i\in S_{m_i}$, in the case that  $\{\eta_i\}_{i=1}^n$ are all simple permutations. A particular case of this is the set of separable permutations, which corresponds to $r=2$, $\eta_1=2413,\eta_2=3142$.

\end{abstract}

\maketitle
\section{Introduction and Statement of Results}\label{intro}
Let $k,n\in\mathbb{N}$ with $2\le k\le n$.
Let $P_n$ denote the uniform probability measure on the set $S_n$ of permutations of $[n]:=\{1,\cdots, n\}$ and denote   by $E_n$ expectations corresponding to $P_n$.  Denote a permutation $\sigma\in S_n$ by $\sigma=\sigma_1\sigma_2\cdots \sigma_n$.
The set of  $k$ consecutive numbers $\{l,l+1,\cdots, l+k-1\}\subset [n]$ appears
in a  set  of consecutive positions in the permutation if there exists an $m$ such that $\{l,l+1,\cdots, l+k-1\}=\{\sigma_m,\sigma_{m+1},\cdots, \sigma_{m+k-1}\}$.
Let $A_{n}^{k;l}\subset S_n$ denote the event that
 the set of  $k$ consecutive numbers $\{l, l+1,\cdots, l+k-1\}$ appears
in a  set  of consecutive positions.
It is immediate that for any $1\le l,m\le n-k+1$,
the probability that  $\{l,l+1,\cdots, l+k-1\}=\{\sigma_m,\sigma_{m+1},\cdots, \sigma_{m+k-1}\}$ is equal to $\frac{k!(n-k)!}{n!}$.
Thus,
\begin{equation}\label{uniformprob}
P_n(A_{n}^{k;l})=(n-k+1)\frac{k!(n-k)!}{n!}\sim\frac{k!}{n^{k-1}},\ \text{as}\ n\to\infty, \text{for}\ k\ge2.
\end{equation}
Let
$N_{n}^{(k)}=\sum_{l=1}^{n-k+1}1_{A_{n}^{k;l}}$ denote the number of sets of $k$ consecutive numbers
appearing in sets of consecutive positions, and let $A_{n}^k=\cup_{l=1}^{n-k+1}A_{n}^{k;l}$ denote the event that there exists a set of   $k$ consecutive numbers
appearing in a  set of consecutive positions.
Then
\begin{equation}\label{uniformexpect}
E_nN_{n}^{(k)}=(n-k+1)^2\frac{k!(n-k)!}{n!}\sim\frac{k!}{n^{k-2}},\ \text{as}\ n\to\infty,\ \text{for}\ k\ge2.
\end{equation}
Using inclusion-exclusion along with \eqref{uniformprob},
it is not hard to show that
\begin{equation}\label{uniformtotalprob}
P_n(A_{n}^k)\sim\frac{k!}{n^{k-2}},\ \text{as}\ n\to\infty,\ \text{for}\ k\ge3.
\end{equation}
It follows from \eqref{uniformexpect}  (or from \eqref{uniformtotalprob}) that for $k\ge3$, the sequence $\{N^{(k)}_{n}\}_{n=1}^\infty$ converges to zero in probability. On the other hand,
  $\{N^{(2)}_{n}\}_{n=1}^\infty$ converges in distribution to a Poisson random variable with parameter $2$. This result
 on the clustering of consecutive numbers in permutations
 goes back over 75 years; see
\cite{W}, \cite{K}.

In this article, we study the clustering of consecutive numbers in permutations avoiding a pattern of length three, as well as in permutations that avoid a fixed number of patterns, all of which are \it simple,\rm\ the definition of which is given below.
We recall the definition of pattern avoidance for permutations.
If $\sigma=\sigma_1\sigma_2\cdots\sigma_n\in S_n$ and $\eta=\eta_1\cdots\eta_m\in S_m$, where $2\le m\le n$,
then we say that $\sigma$ contains $\eta$ as a pattern if there exists a subsequence $1\le i_1<i_2<\cdots<i_m\le n$ such
that for all $1\le j,k\le m$, the inequality $\sigma_{i_j}<\sigma_{i_k}$ holds if and only if the inequality $\eta_j<\eta_k$ holds.
If $\sigma$ does not contain $\eta$, then we say that $\sigma$ \it avoids\rm\ $\eta$.
We denote by $S_n^{\text{av}(\eta)}$ the set of permutations in $S_n$ that avoid $\eta$.
If $n<m$, we define $S_n^{\text{av}(\eta)}=S_n$. We denote by
$P_n^{\text{av}(\eta)}$ the uniform probability measure on $S_n^{\text{av}(\eta)}$, and denote by
$E_n^{\text{av}(\eta)}$ expectations corresponding to $P_n^{\text{av}(\eta)}$.

The main results in this paper concern permutations avoiding a pattern of length three.
Let $\eta\in S_3$, $k\ge 2$ and $\tau\in S_k^{\text{av}(\eta)}$.
For $\sigma\in S_n^{\text{av}(\eta)}$, with $n\ge k$,
let $N_n^{(k;\tau)}(\sigma)$ denote the number of occurrences  of $k$ consecutive numbers appearing in $k$ consecutive positions in $\sigma$,  and such that the order of their appearance is the pattern $\tau$.
(Such an occurrence is defined by the existence of
  $1\le l,m\le n-k+1$ such that
 $\{l, l+1,\cdots, l+k-1\}=\{\sigma_m,\cdots,\sigma_{m+k-1}\}$ and in addition,
 $\tau_i=\sigma_{m+i-1}-(l-1)$, $i=1,\cdots, k$.) Let
 $N_n^{(k)}(\sigma)$ denote the number of occurrences of $k$ consecutive numbers appearing in $k$ consecutive positions in $\sigma$, without regard to the order of their appearance;
 that is,  $N_n^{(k)}(\sigma)=\sum_{\tau\in S_k^{\text{av}(\eta)}}N_n^{(k;\tau)}(\sigma)$.
 We obtain explicit  formulas for
$E_n^{\text{av}(\eta)}N_n^{(k;\tau)}$ and $E_n^{\text{av}(\eta)}N_n^{(k)}$, for all $2\le k\le n$, all $\eta\in S_3$ and all $\tau\in S_k^{\text{av}(\eta)}$.
(Of course, $P_n^{\text{av}(\eta)}(N_n^{(k;\tau)}=0)=1$, if $\tau\in S_k-S_k^{\text{av}(\eta)}$.)
These exact formulas then yield  asymptotic formulas
as $n\to\infty$ with $k$ fixed,  and as $n\to\infty$ with  $k=k_n\to\infty$.
We also obtain a law of large numbers for $N_n^{(k;\tau)}$ and $N_n^{(k)}$ when $\eta\in S_3-\{123,321\}$.
It turns out that for $\eta\in S_3-\{123,321\}$ and all $\tau\in S_k$, $E_n^{\text{av}(\eta)}N_n^{(k;\tau)}$ grows linearly in $n$.
However, for each  $\eta\in\{123,321\}$,
there is one distinguished $\tau\in S_k$ for which
 $E_n^{\text{av}(\eta)}N_n^{(k;\tau)}$ has linear growth in $n$, while for all other $\tau\in S_k$,
$E_n^{\text{av}(\eta)}N_n^{(k;\tau)}$ converges to a positive constant.

Although there are six permutations $\eta$ in $S_3$, it suffices to consider just two of them---one from $\{132, 213,231,312\}$ and one from $\{123,321\}$.
Indeed, recall that the \it reverse\rm\  of a permutation $\sigma=\sigma_1\cdots\sigma_n$ is the permutation $\sigma^{\text{rev}}:=\sigma_n\cdots\sigma_1$,
and the \it complement\rm\ of $\sigma$ is the permutation
$\sigma^{\text{com}}$ satisfying
$\sigma^{\text{com}}_i=n+1-\sigma_i,\ i=1,\cdots, n$.
Let $\sigma^{\text{rev-com}}$ denote the permutation obtained by applying reversal and then complementation to $\sigma$ (or equivalently, vice versa).
It is immediate that the quantity  $E_n^{\text{av}(\eta)}N_n^{(k;\tau)}$ remains unchanged if each of $\eta$ and $\tau$ is replaced by its reversal, or by its complementation, or by its
complementation-reversal. Thus, it suffices to consider, say $231$ and $321$, since
 the three permutations $132$, $213$ and $312$ are obtained from $231$ respectively
by reversal, complementation and complementaion-reversal, and the permutation $123$ is obtained from $321$ by reversal.
We will prove  our results for the patterns $231$ and $321$, but we state them  in complete generality.

Denote the $n$th Catalan number by $C_n$: $C_n=\frac1{n+1}\binom{2n}n$. As is well known, $|S_n^{\text{av}(\eta)}|=C_n$, for all $\eta\in S_3$ and all $n\in\mathbb{N}$ \cite{B,SS}.
\begin{theorem}\label{nonmonthm}
Let $\eta\in\{132,213,231,312\}$.

\noindent i. Let $2\le k\le n$ and define
$\tau^*(\eta)\in S_k^{\text{av}(\eta)}$ by
\begin{equation}\label{*}
\tau^*(\eta)=\begin{cases} k\cdots 1,\ \text{if}\ \eta=231\ \text{or}\ 312;\\ 1\cdots k,\ \text{if}\ \eta=132\ \text{or}\ \eta=213.\end{cases}
\end{equation}
For $\tau\in S_k^{\text{av}(\eta)}$,
\begin{equation}\label{nonmonformula}
E_n^{\text{av}(\eta)}N_n^{(k;\tau)} =\begin{cases} \frac{(n-k+2)C_{n-k+1}}{2C_n}, \text{if}\ \tau\neq \tau^*(\eta);\\
\frac{(n-k+3)C_{n-k+2}}{2C_n}-\frac{(n-k+2)C_{n-k+1}}{C_n}, \text{if}\ \tau=\tau^*(\eta).
\end{cases}.
\end{equation}
Also,
\begin{equation}\label{WLLN}
\lim_{n\to\infty}\frac{N_n^{(k;\tau)}}{E_n^{\text{av}(\eta)}N_n^{(k;\tau)}}=1\ \text{in probability}.
\end{equation}
\noindent ii.
\begin{equation}\label{fullnonmon}
E_n^{\text{av}(\eta)}N_n^{(k)}=\frac{C_{n-k+1}}{2C_n}\Big((n-k+2)C_k+n-k\Big).
\end{equation}
\end{theorem}
Also,
\begin{equation}\label{WLLNfull}
\lim_{n\to\infty}\frac{N_n^{(k)}}{E_n^{\text{av}(\eta)}N_n^{(k)}}=1\ \text{in probability}.
\end{equation}
\bf\noindent Remark.\rm\ It is easy to check that $E_n^{\text{av}(\eta)}N_n^{(k;\tau)}$ is larger for $\tau=\tau^*(\eta)$ than for $\tau\neq\tau^*(\eta)$,
for all $n>k$. (For $n=k$, they are both of course equal to $\frac1{C_k}$.)

\medskip

Using the fact that $C_n\sim\frac{4^n}{\sqrt\pi n^\frac32}$,
the following corollary of Theorem \ref{nonmonthm} follows by straightforward calculation.

\begin{corollary}\label{nonmoncor}
Let $\eta\in\{132,213,231,312\}$ and let $\tau^*(\eta)$ be as in \eqref{*}.

\noindent i. Let $\tau^*(\eta)\neq \tau\in S_k^{\text{av}(\eta)}$. Then
\begin{equation}\label{asymplim}
\begin{aligned}
&\lim_{n\to\infty}\frac1nE_n^{\text{av}(\eta)}N_n^{(k;\tau)}=\frac1{2\cdot4^{k-1}};\\
&E_n^{\text{av}(\eta)}N_n^{(k_n;\tau)}\sim\frac n{2\cdot4^{k_n-1}},\ \text{if}\ k_n=o(n).
\end{aligned}
\end{equation}
ii. Let $\tau=\tau^*(\eta)$. Then
\begin{equation}\label{asymplim*}
\begin{aligned}
&\lim_{n\to\infty}\frac1nE_n^{\text{av}(\eta)}N_n^{(k;\tau)}=\frac1{4^{k-1}};\\
&E_n^{\text{av}(\eta)}N_n^{(k_n;\tau)}\sim\frac n{4^{k_n-1}},\ \text{if}\ k_n=o(n).
\end{aligned}
\end{equation}
\noindent iii.
\begin{equation}\label{asymplimalltau}
\begin{aligned}
&\lim_{n\to\infty}\frac1nE_n^{\text{av}(\eta)}N_n^{(k)}=\frac{C_k+1}{2\cdot 4^{k-1}};\\
&E_n^{\text{av}(\eta)}N_n^{(k_n)}\sim\frac2{\sqrt\pi}\frac n{k_n^\frac32},\ \text{if}\  \lim_{n\to\infty}k_n=\infty\ \text{and}\ k_n=o(n).
\end{aligned}
\end{equation}
\end{corollary}
\noindent \bf Remark.\rm\ From \eqref{asymplimalltau}, it follows that $E_n^{\text{av}(\eta)}N_n^{(k_n)}$ remains bounded away from zero when $n\to\infty$ if and only if
$k_n=O(n^\frac23)$, and from \eqref{asymplim} and \eqref{asymplim*} it follows that  $E_n^{\text{av}(\eta)}N_n^{(k_n;\tau)}$
remains bounded away from zero when $n\to\infty$ if and only if
$\limsup_{n\to\infty}(k_n-\frac{\log n}{\log 4})<\infty$.
\begin{theorem}\label{monthm}
Let $\eta\in\{123,321\}$.

\noindent i. Let $2\le k\le n$ and define
$\tau^*(\eta)\in S_k^{\text{av}(\eta)}$ by
\begin{equation}\label{*again}
\tau^*(\eta)=\begin{cases} k\cdots 1,\ \text{if}\ \eta=123;\\ 1\cdots k,\ \text{if}\ \eta=321.\end{cases}
\end{equation}
For $\tau\in S_k^{\text{av}(\eta)}$,
\begin{equation}\label{monformula}
E_n^{\text{av}(\eta)}N_n^{(k;\tau)}
 =\begin{cases}\frac{C_{n-k+1}}{C_n},\ \text{if}\ \tau\neq\tau^*(\eta);\\ (n-k+1)\frac{C_{n-k+1}}{C_n},\ \text{if}\ \tau=\tau^*(\eta).\end{cases}
\end{equation}
\noindent ii.
\begin{equation}\label{fullmon}
E_n^{\text{av}(\eta)}N_n^{(k)}=\frac{C_{n-k+1}(n-k+C_k)}{C_n}.
\end{equation}
\end{theorem}
\noindent\bf Remark 1.\rm\ From \eqref{nonmonformula} and \eqref{monformula}, it follows that for each $k\ge2$,
$E_n^{\text{av}(\eta)}N_n^{(k;\tau)}$ has linear growth in $n$, for all pairs $(\eta,\tau)$, with   $\eta\in \{132, 213, 231, 312\}$ and $\tau\in S_k^{\text{av}(\eta)}$,
and for the pairs $\eta=123,\tau=k\cdots 1$ and $\eta=321,\tau=1\cdots k$.
On the other hand
$E_n^{\text{av}(\eta)}N_n^{(k;\tau)}$  is bounded and bounded away from zero for $\eta=123$ and $k\cdots 1\neq\tau\in S_k^{\text{av}(123)}$ and
for $\eta=321$ and $1\cdots k\neq\tau\in S_k^{\text{av}(321)}$.
\medskip

\noindent\bf Remark 2.\rm\ We elaborate on the behavior in the case of the pairs $\eta=123$ and $kk-1\cdots1\neq \tau\in S_k^{\text{av}(123)}$,
and the pairs $\eta=321$ and $12\cdots k\neq \tau\in S_k^{\text{av}(321)}$. Consider, for example, $\eta=321$ and $12\cdots k\neq \tau\in S_k^{\text{av}(321)}$.  As is well known,
every permutation in $S_n^{\text{av}(321)}$ is composed of two increasing subsequences. (The two subsequences are not necessarily unique; for example,
the permutation $145236798\in S_9^{\text{av}(321)}$ is composed of the two increasing subsequence 12379 and 4568 as well as of the two increasing subsequences
12368 and 4579.) A cluster of length $k$ and pattern $\tau$ in $\sigma\in S_n^{\text{av}(321)}$ will necessarily
have to include numbers from both increasing subsequences corresponding to $\sigma$. Theorem \ref{monthm} states that the expected number of such clusters is bounded
(and bounded away from zero) as $n\to\infty$. This means that the two subsequences have very little such intertwining.
\medskip

The following corollary of Theorem \ref{monthm} follows by straightforward calculation.
\begin{corollary}\label{moncor}
Let $\eta\in\{123,321\}$ and let $\tau^*(\eta)$ be as in \eqref{*again}.

\noindent i. Let $\tau^*(\eta)\neq \tau\in S_k^{\text{av}(\eta)}$. Then
\begin{equation*}\label{asymplimagain}
\begin{aligned}
&\lim_{n\to\infty}E_n^{\text{av}(\eta)}N_n^{(k;\tau)}=\frac1{4^{k-1}};\\
&E_n^{\text{av}(\eta)}N_n^{(k_n;\tau)}\sim\frac1{4^{k_n-1}},\ \text{if}\ k_n=o(n).
\end{aligned}
\end{equation*}
\noindent ii. Let $\tau=\tau^*(\eta)$. Then
\begin{equation}\label{asymplim*again}
\begin{aligned}
&\lim_{n\to\infty}\frac1nE_n^{\text{av}(\eta)}N_n^{(k;\tau)}=\frac1{4^{k-1}};\\
&E_n^{\text{av}(\eta)}N_n^{(k_n;\tau)}\sim\frac n{4^{k_n-1}},\ \text{if}\ k_n=o(n).
\end{aligned}
\end{equation}
iii.
\begin{equation}\label{asymplimalltauagain}
\begin{aligned}
&\lim_{n\to\infty}\frac1nE_n^{\text{av}(\eta)}N_n^{(k)}=\frac1{4^{k-1}};\\
&E_n^{\text{av}(\eta)}N_n^{(k)}\sim\frac n{4^{k_n-1}}+\frac4{\sqrt\pi k_n^\frac32},\ \text{if}\ \lim_{n\to\infty}k_n=\infty\ \text{and}\ k_n=o(n).
\end{aligned}
\end{equation}
\end{corollary}
\noindent \bf Remark.\rm\ From \eqref{asymplim*again} and \eqref{asymplimalltauagain}, it follows that each of  $E_n^{\text{av}(\eta)}N_n^{(k_n)}$
and $E_n^{\text{av}(\eta)}N_n^{(k_n;\tau^*(\eta))}$
remains bounded away from zero when $n\to\infty$ if and only if
$\limsup_{n\to\infty}(k_n-\frac{\log n}{\log 4})<\infty$.
\medskip

The methods of proof for Theorem \ref{nonmonthm} and Theorem \ref{monthm} are completely different from one another.
The proof of Theorem \ref{nonmonthm} exploits generating functions, whereas the proof of Theorem \ref{monthm} is much more of a purely combinatorial argument.
The method of proof of Theorem \ref{monthm} easily extends to  allow one to obtain similar results for permutations avoiding simple patterns.
A permutation $\eta\in S_m$ is called \it simple\rm\ if $\eta\not\in A^{(m)}_{k;l}$, for all $k\in\{2,\cdots, m-1\}$ and all $l\in\mathbb{N}$ satisfying
$k+l-1\le m$. (Equivalently, $\eta$ is simple if and only if $\{\eta_a,\cdots,\eta_{a+k-1}\}$ is not equal to a block of $k$ consecutive numbers in $[m]$, for all
$k\in\{2,\cdots, m-1\}$ and all $a\in\mathbb{N}$ satisfying $a+k-1\le m$.)
For example, 6241753 is a simple permutation in $S_7$, but 6241375 is not, because of the block 2413.
Note that there are no simple permutations in $S_2$ or $S_3$.

For $r\in\mathbb{N}$ and a collection of permutations $\{\eta_i\}_{i=1}^r$, with $\eta_i\in S_{m_i}$, where $m_i\ge2$, denote by
$S_n^{\text{av}(\eta_1,\cdots,\eta_r)}$ the set of permutations in $S_n$ that avoid all of the patterns $\{\eta_i\}_{i=1}^r$.
Denote by
$P_n^{\text{av}(\eta_1,\cdots,\eta_r)}$ the uniform probability measure on $S_n^{\text{av}(\eta_1,\cdots,\eta_r)}$, and denote by
$E_n^{\text{av}(\eta_1,\cdots,\eta_r)}$ expectations corresponding to $P_n^{\text{av}(\eta_1,\cdots,\eta_r)}$.
For $\sigma\in S_n^{\text{av}(\eta_1,\cdots,\eta_r)}$ and  $\tau\in S_k^{\text{av}(\eta_1,\cdots,\eta_r)}$, with $2\le k\le n$,
let $N_n^{(k;\tau)}(\sigma)$ denote the number of occurrences  of $k$ consecutive numbers appearing in $k$
consecutive positions in $\sigma$,  and such that the order of their appearance is the pattern
$\tau$.
Let $N_n^{(k)}(\sigma)$ denote the number of occurrences of $k$ consecutive numbers appearing in $k$ consecutive positions in $\sigma$, without regard to the order of their appearance.
\begin{theorem}\label{multipatternthm} Let $r\in\mathbb{N}$ and let $\{\eta_i\}_{i=1}^r$, with $\eta_i\in S_{m_i}$ and $m_i\ge4$, be  simple permutations. Then for
$2\le k\le n$ and $\tau\in S_k^{\text{av}(\eta_1,\cdots,\eta_m)}$,
  \begin{equation}\label{multipattern}
E_n^{\text{av}(\eta_1,\cdots,\eta_r)}N_n^{(k;\tau)}=
\frac{(n-k+1)|S^{\text{av}(\eta_1,\cdots,\eta_r)}_{n-k+1}|}
{|S_n^{\text{av}(\eta_1,\cdots,\eta_m)}|}.
\end{equation}
Also,
\begin{equation}\label{multipatternall}
E_n^{\text{av}(\eta_1,\cdots,\eta_m)}N_n^{(k)}=\frac{(n-k+1)|S^{\text{av}(\eta_1,\cdots,\eta_r)}_{n-k+1}|\thinspace|S_k^{\text{av}(\eta_1,\cdots,\eta_m)}|}
{|S_n^{\text{av}(\eta_1,\cdots,\eta_m)}|}.
\end{equation}
\end{theorem}
We elaborate  on two particular cases of Theorem \ref{multipatternthm}.

A \it separable\rm\ permutation is a permutation that can be constructed from the  singleton in $S_1$ via a series
of  iterations of
\it direct sums \rm\ and \it skew sums\rm. (See \cite{P21}, for example,  for more details.) An equivalent definition
of a separable permutation \cite{BBL} is a permutation that avoids the two patterns 2413 and 3142.
The generating function for the enumeration of separable permutations is known explicitly and allows one
to show \cite[p. 474-475]{FS} that
\begin{equation}\label{sepasymp}
|S_n^{\text{\rm sep}}|\sim\frac1{2\sqrt{\pi n^3}}(3-2\sqrt2)^{-n+\frac12}.
\end{equation}
(The sequence on the right hand side of \eqref{sepasymp} is the sequence of big Schr\"oder numbers--A006318 in the \it On-Line Encyclopedia of Integer Sequences.)\rm\
Let $S_n^{\text{sep}}$ denote the set of separable permutations in $S_n$, let
$P_n^{\text{sep}}$ denote the uniform probability measure on $S_n^{\text{sep}}$ and let $E_n^{\text{sep}}$ denote the corresponding expectation.
Since the permutations 2413 and 3142 are simple, the following result follows immediately from Theorem \ref{multipatternthm} and \eqref{sepasymp}.
\begin{corollary}\label{separablecor}
For
$2\le k\le n$ and $\tau\in S_k^{\text{sep}}$,
\begin{equation*}\label{separable}
E_n^{\text{sep}}N_n^{(k;\tau)}=\frac{(n-k+1)|S_{n-k+1}^{\text{sep}}|}{|S_n^{\text{sep}}|},
\end{equation*}
and
\begin{equation*}\label{separableasym}
\lim_{n\to\infty}\frac{E_n^{\text{sep}}N_n^{(k;\tau)}}n=(3-2\sqrt2)^{k-1}.
\end{equation*}
Also,
\begin{equation*}\label{separableall}
E_n^{\text{sep}}N_n^{(k)}=
\frac{(n-k+1)|S_{n-k+1}^{\text{sep}}|\thinspace|S_k^{\text{sep}}|}{|S_n^{\text{sep}}|},
\end{equation*}
and
\begin{equation*}\label{separableasympall}
\lim_{n\to\infty}\frac{E_n^{\text{sep}}N_n^{(k)}}n=(3-2\sqrt2)^{k-1}|S_k^{\text{sep}}|.
\end{equation*}
\end{corollary}

One formulation of the Stanley-Wilf conjecture, completely proved in \cite{MT}, states that for every permutation $\tau\in S_m$, $m\ge2$, there exists a number
$L(\tau)>1$ such that
\begin{equation*}\label{SW}
\lim_{n\to\infty}|S_n(\tau)|^\frac1n=L(\tau).
\end{equation*}
We refer to $L(\tau)$ as the \it Stanley-Wilf limit.\rm\
Furthermore, it is known that
\begin{equation}\label{ratio}
\lim_{n\to\infty}\frac{|S_{n+1}(\tau)|}{|S_n(\tau)|}=L(\tau),\ \text{for any simple permutation}\ \tau\in S_k,\thinspace k\ge2.
\end{equation}
Indeed, \eqref{ratio} was proven in   \cite{AM}  for a wide class of permutations $\tau$, including all simple ones.
Thus, the following asymptotic result  follows as an immediate corollary of Theorem \ref{multipatternthm} and \eqref{ratio}.
\begin{corollary}\label{onepermavoiding}
Let   $\eta\in S_m$,\ $m\ge4$, be a simple permutation. Then for
$2\le k\le n$ and $\tau\in S_k^{\text{av}(\eta)}$,
  \begin{equation*}\label{multipatternone}
\lim_{n\to\infty}\frac1nE_n^{\text{av}(\eta)}N_n^{(k;\tau)}=\frac1{L(\tau)^{k-1}},
\end{equation*}
and
\begin{equation*}\label{multipatternallone}
\lim_{n\to\infty}\frac1nE_n^{\text{av}(\eta)}N_n^{(k)}=\frac{|S_k^{\text{av}(\eta)}|}{L(\tau)^{k-1}},
\end{equation*}
where $L(\tau)$ is the Stanley-Wilf limit.
\end{corollary}

This leads to  a natural  question.

\bf\noindent Open Question.\rm\  Is it true that for every $\eta\in \cup_{j=1}^\infty S_j$ and all $k\ge2$,
the expectation $E_n^{\text{av}(\eta)}N_n^{(k)}$   grows linearly  in $n$,  or equivalently, that there exists a
$\tau\in S_k^{\text{av}(\eta)}$ such that the expectation
$E_n^{\text{av}(\eta)}N_n^{(k;\tau)}$ grows linearly in $n$?
\medskip

The recent paper \cite{P22} studied the clustering of consecutive numbers under Mallows distributions.

In section \ref{thm1proof}, we prove Theorem \ref{nonmonthm}, except for the law of large numbers in \eqref{WLLN} and \eqref{WLLNfull}. Of course, \eqref{WLLNfull} follows immediately from \eqref{WLLN}.
Using the second moment method, \eqref{WLLN} follows immediately from Corollary \ref{nonmoncor} and the following proposition, which we prove in section \ref{varproof}.
\begin{proposition}\label{nonmonpropvar}
Under the assumptions in Theorem \ref{nonmonthm}, for all $\tau\in S_k^{\text{av}(\eta)}$,
$\text{Var}(N_n^{(k;\tau)})=o(n^2)$.
\end{proposition}

We prove  Theorem \ref{monthm} in section \ref{thm2proof} and Theorem \ref{multipatternthm} in section \ref{thm3proof}.

\section{Proof of Theorem \ref{nonmonthm}}\label{thm1proof}
For the proof, it will be convenient to define
 $N_m^{(k;\tau)}\equiv0$, if $m<k$.
As noted before the statement of the theorem, it suffices to consider the case $\eta=231$. Let $k\ge2$ and $\tau\in S_k^{\text{av}(231)}$. We consider $n\ge k$.
The following definition will be useful. Let $a_1<a_2\cdots <a_m$ be real numbers and let $\rho=\rho_1\cdots\rho_m$ be a permutation of these numbers.
We define $\text{red}(\rho)\in S_m$, the reduction of $\rho$, to be the permutation in $S_m$ that has the same pattern as $\rho$. That is,
$\text{red}(\rho)=\sigma$ if $\sigma$ satisfies $\sigma_i<\sigma_j$ whenever $\rho_i<\rho_j$, $i,j\in[m]$.

Every permutation $\sigma\in S_k^{\text{av}(231)}$  has the property that if $\sigma_j=n$, then the numbers $\{1,\cdots, j-1\}$ appear in the first $j-1$ positions in $\sigma$ (and then of course, the numbers $\{j,\cdots, n-1\}$ appear in the last $n-j$ positions in $\sigma$.)
From this fact, along with the fact that $|S_n^{\text{av}(\eta)}|=C_n$, it follows that
\begin{equation}\label{nprob}
P_n^{\text{av}(231)}(\sigma_j=n)=\frac{C_{j-1}C_{n-j}}{C_n},\ \text{for}\ j\in[n].
\end{equation}
It also follows that under the conditioned  measure $P_n^{\text{av}(231)}|\{\sigma_j=n\}$, the permutation $\sigma_1\cdots\sigma_{j-1}\in S_{j-1}$ has the distribution
$P_{j-1}^{\text{av}(231)}$,  the permutation $\text{red}(\sigma_{j+1}\cdots\sigma_n)$ has the distribution $P_{n-j}^{\text{av}(231)}$, and these two permutations are independent.
The facts in this paragraph will be used frequently in the proof, sometimes without comment and sometimes with a reference to ``the second paragraph of section \ref{thm1proof}.''

Note that the following well-known recursion formula for the Catalan numbers follows from \eqref{nprob}.
\begin{equation}\label{catrecur}
C_n=\sum_{j=1}^nC_{j-1}C_{n-j},\ n\in\mathbb{N}.
\end{equation}
The key to proving the theorem is the following proposition, whose rather long technical proof will be postponed until we have completed the proof of the theorem.
\begin{proposition}\label{propcond}
Let $2\le k\le n$ and let $\tau\in S_k^{\text{av}(231)}$.
For each $m\in\mathbb{N}$, each $\sigma\in S_m^{\text{av}(231)}$ and each $A\subset S_m^{\text{av}(231)}$,
let the random variables $N_m^{(k;\tau)}(\cdot)$,  $1_{\sigma}(\cdot)$ and $1_A(\cdot)$
 be defined on the probability space $(S_m,P_m^{\text{av}(231)})$.

\noindent i. Let $k\thinspace k-1\thinspace\cdots1\neq\tau\in S_k^{\text{av}(231)}$. Denote $i_k=\tau^{-1}_k$. Then
\begin{equation}\label{propcondgeneralnot1}
N_n^{(k;\tau)}|\{\sigma_j=n\}\stackrel{\text{dist}}{=}
N_{j-1}^{(k;\tau)}+N_{n-j}^{(k;\tau)}+1_{\{j=n-k+i_k\}}1_{A^{l}_{n-k+i_k-1}}1_{\sigma^{*,k-i_k}},
\end{equation}
where
\begin{equation}\label{sigma*first}
\sigma^{*,k-i_k}=\text{red}(\tau_{i_k+1}\tau_{i_k+2}\cdots\tau_k)\in S_{k-i_k}^{\text{av}(231)}
\end{equation}
 and
\begin{equation}\label{propAl}
A^{l}_{n-k+i_k-1}=\{\sigma\in S_{n-k+i_k-1}^{\text{av}(231)}: \sigma_{n-k+l}=\tau_l+n-k,\ l=1\cdots i_k-1\},
\end{equation}
and where $N_{j-1}^{(k,\tau)}$ is independent of $N_{n-j}^{(k,\tau)}$, and the pair $N_{n-k+i_k-1}^{(k,\tau)}, 1_{A^{l}_{n-k+i_k-1}}$ is independent of the pair $N_{k-i_k}^{(k,\tau)},1_{\sigma^{*,k-i_k}}$.
If $i_k=k$, then we understand $1_{\sigma^{*,k-i_k}}$ to be the constant 1, and if $i_k=1$, then we understand $1_{A^{l}_{n-k+i_k-1}}$ to be the constant 1.

\noindent ii. Let $\tau=k\thinspace k-1\thinspace\cdots1$. Then
\begin{equation}\label{propcondgeneral1mon}
N_n^{(k;\tau)}|\{\sigma_j=n\}\stackrel{\text{dist}}{=}
N_{j-1}^{(k;\tau)}+N_{n-j}^{(k;\tau)}+1_{\{j\le n-k+1\}}1_{A^{r}_{n-j}},
\end{equation}
where
\begin{equation}\label{propAr}
A^{r}_{n-j}=\{\sigma\in S_{n-j}:\sigma_l=n-j+1-l,\ l=1,\cdots, k-1\},
\end{equation}
and where $N_{j-1}^{(k,\tau)}$ is independent of the pair $N_{n-j}^{(k,\tau)},1_{A_{n-j}}$.
\end{proposition}
(The superscripts $l$ and $r$ in \eqref{propAl} and \eqref{propAr} are just used to distinguish the two sets, and stand for ``left'' and ``right'', because of how they arise in the proof.)

Let $s^\tau_n=E_n^{\text{av}(231)}N_n^{(k;\tau)}$.
Then $s^\tau_n=0$, for $n=1,\cdots, k-1$. For convenience, define $s^\tau_0=0$. The following result follows easily from Proposition \ref{propcond}.
\begin{proposition}\label{recursionprop}
For $2\le k\le n$,
\begin{equation}\label{snrecur}
s^\tau_n=\begin{cases}2\sum_{j=1}^n\frac{C_{j-1}C_{n-j}}{C_n}s^\tau_{j-1}+\frac{C_{n-k}}{C_n},\ \text{if}\  k\thinspace k-1\thinspace\cdots 1\neq\tau\in S_k^{\text{av}(231)};\\
2\sum_{j=1}^n\frac{C_{j-1}C_{n-j}}{C_n}s^\tau_{j-1}+\frac{C_{n-k+1}}{C_n},\ \text{if}\ \tau=k\thinspace k-1\thinspace\cdots 1.\end{cases}
\end{equation}
\end{proposition}
\begin{proof}
By definition, $\tau_{i_k}=k$. Thus, since $\tau\in S_k^{\text{av}(231)}$, it follows from the second paragraph of section \ref{thm1proof} that $\{\tau_1,\cdots, \tau_{i_k-1}\}=\{1,\cdots, i_k-1\}$.
Using this with  the second paragraph of section \ref{thm1proof} and \eqref{propAl},  we obtain the first equality below, and using  the second paragraph of section \ref{thm1proof} and  \eqref{propAr}, we obtain the second equality below:
\begin{equation}\label{probAlAr}
P_n^{\text{av}(231)}(A^l_{n-k+i_k-1})= \frac{C_{n-k}}{C_{n-k+i_k-1}};\ \ P_n^{\text{av}(231)}(A^r_{n-j})=\frac{C_{n-j-k+1}}{C_{n-j}}.
\end{equation}
From \eqref{nprob}, \eqref{propcondgeneralnot1}-\eqref{propAl} and \eqref{probAlAr}, we have
\begin{equation*}
\begin{aligned}
&s^\tau_n=\sum_{j=1}^n\frac{C_{j-1}C_{n-j}}{C_n}(s^\tau_{j-1}+s^\tau_{n-j})+\frac{C_{n-k+i_k-1}C_{k-i_k}}{C_n}\frac{C_{n-k}}{C_{n-k+i_k-1}}\frac1{C_{k-i_k}}=\\
&2\sum_{j=1}^n\frac{C_{j-1}C_{n-j}}{C_n}s^\tau_{j-1}+\frac{C_{n-k}}{C_n},\ \text{for}\ k\thinspace k-1\cdots1\neq\tau\in S_k^{\text{av}(231)},
\end{aligned}
\end{equation*}
and from
\eqref{nprob}, \eqref{propcondgeneral1mon} and \eqref{probAlAr}, we have
\begin{equation*}
\begin{aligned}
&s^\tau_n=\sum_{j=1}^n\frac{C_{j-1}C_{n-j}}{C_n}(s^\tau_{j-1}+s^\tau_{n-j})+\sum_{j=1}^{n-k+1}\frac{C_{j-1}C_{n-j}}{C_n}\frac{C_{n-j-k+1}}{C_{n-j}}=\\
&2\sum_{j=1}^n\frac{C_{j-1}C_{n-j}}{C_n}s^\tau_{j-1}+\frac{C_{n-k+1}}{C_n},\ \text{for}\ \tau=k\thinspace k-1\cdots1,
\end{aligned}
\end{equation*}
where the last equality follows from \eqref{catrecur}.
\end{proof}

Define
\begin{equation}\label{G}
G^{(k,\tau)}(t)=\sum_{j=k}^\infty C_js^\tau_jt^j,
\end{equation}
and let
\begin{equation*}\label{Catalangen}
C(t)=\sum_{j=0}^\infty C_jt^j
\end{equation*}
denote the generating function of the Catalan numbers.
As is well-known \cite{P14},
\begin{equation}\label{Catalangenexpl}
C(t)=\frac{1-\sqrt{1-4t}}{2t},\ |t|\le \frac14.
\end{equation}
We use Proposition \ref{recursionprop} to prove the following result.
\begin{proposition}\label{propG} For $|t|\le\frac14$,
\begin{equation} \label{formulaG}
G^{(k,\tau)}(t)=\begin{cases}\frac{t^kC(t)}{1-2tC(t)}, \ \text{if}\ k\thinspace k-1\cdots 1\neq\tau\in S_k^{\text{av}(231)};\\
\frac{t^{k-1}(C(t)-1)}{1-2tC(t)}, \ \text{if}\ \tau=k\thinspace k-1\cdots 1.\end{cases}
\end{equation}
\end{proposition}
\begin{proof}
Multiplying the first line of \eqref{snrecur} by $C_nt^n$ and summing over $n$ from $k$ to $\infty$ gives
$$
G^{(k,\tau)}(t)=2tC(t)G^{(k,\tau)}(t)+t^kC(t),\ k\thinspace k-1\cdots 1\neq\tau\in S_k,
$$
and solving above for $G^{(k,\tau)}(t)$ gives  \eqref{formulaG} for $\tau\neq k\thinspace k-1\cdots 1$.
Multiplying the second line of  \eqref{snrecur} by $C_nt^n$ and summing over $n$ from $k$ to $\infty$ gives
$$
G^{(k,\tau)}(t)=2tC(t)G^{(k,\tau)}(t)+t^{k-1}(C(t)-1),\ \tau=k\thinspace k-1\cdots 1,
$$
and solving above for $G^{(k,\tau)}(t)$ gives  \eqref{formulaG} for $\tau=k\thinspace k-1\cdots 1$.

\end{proof}
We now use Proposition \ref{propG} to prove Theorem \ref{nonmonthm}.

\noindent \it Proof of Theorem \ref{nonmonthm}.
\noindent Part (i).\rm\ From  \eqref{Catalangenexpl}, we have
\begin{equation}\label{Cdivision}
\frac{tC(t)}{1-2tC(t)}=\frac12\big((1-4t)^{-\frac12}-1\big).
\end{equation}
A direct calculation \cite[p.41]{P14} reveals that
\begin{equation}\label{nthderiv+12}
\frac{\big((1-4t)^\frac12\big)^{(n)}|_{t=0}}{n!}=-\frac2{2n-1}\binom{2n-1}n,\ n\ge2.
\end{equation}
From \eqref{nthderiv+12} and the fact that  $\big((1-4t)^\frac12\big)'=-2(1-4t)^{-\frac12}$, it follows  that
\begin{equation}\label{nthderiv-12}
\begin{aligned}
&\frac{\big((1-4t)^{-\frac12}\big))^{(n)}|_{t=0}}{n!}=-\frac12\frac{\big((1-4t)^\frac12\big)^{(n+1)}|_{t=0}}{n!}=\\
&-\frac{n+1}2\frac{\big((1-4t)^\frac12\big)^{(n+1)}|_{t=0}}{(n+1)!}=-\frac{n+1}2\Big(-\frac2{2n+1}\binom{2n+1}{n+1}\Big)=\\
&(n+1)C_n,\ n\ge1.
\end{aligned}
\end{equation}
Now \eqref{Cdivision} and \eqref{nthderiv-12} give
\begin{equation}\label{Cdivisionseries}
\frac{tC(t)}{1-2tC(t)}=\sum_{n=1}^\infty \frac{n+1}2C_nt^n.
\end{equation}

Consider first  the case $k\thinspace k-1\cdots 1\neq\tau\in S_k$.
Then from \eqref{formulaG} and \eqref{Cdivisionseries}
\begin{equation}
G^{(k,\tau)}(t)=\sum_{n=1}^\infty \frac{n+1}2C_nt^{n+k-1},
\end{equation}
which along with \eqref{G} gives
\begin{equation}\label{nonmonfinal231}
E_n^{\text{av}(231)}N_n^{(k;\tau)}=s_n^\tau=\frac{(n-k+2)C_{n-k+1}}{2C_n},
\end{equation}
which is \eqref{nonmonformula} for $\tau\neq k\thinspace k-1\cdots1$ and $\eta=231$.

Now consider the case $\tau=k\thinspace k-1\cdots 1$.
From \eqref{Catalangenexpl}, we have
\begin{equation}\label{1overC}
\frac1{1-2tC(t)}= (1-4t)^{-\frac12}.
\end{equation}
Using this with \eqref{nthderiv-12} gives
\begin{equation}\label{1overCwithk}
\frac{t^{k-1}}{1-2tC(t)}=\sum_{n=0}^\infty(n+1)C_nt^{n+k-1}.
\end{equation}
From \eqref{formulaG},  \eqref{Cdivisionseries} and \eqref{1overCwithk}, we obtain
\begin{equation}\label{finalformulathm1}
G^{(k,\tau)}(t)=\frac{t^{k-1}C(t)}{1-2tC(t)}-
\frac{t^{k-1}}{1-2tC(t)}=\sum_{n=1}^\infty\big(\frac{n+2}2C_{n+1}-(n+1)C_n\big)t^{n+k-1}.
\end{equation}
From \eqref{G} and \eqref{finalformulathm1}, we obtain
\begin{equation}\label{monfinal231}
E_n^{\text{av}(231)}N_n^{(k;\tau)}=s_n^\tau=\frac{(n-k+3)C_{n-k+2}}{2C_n}-\frac{(n-k+2)C_{n-k+1}}{C_n},
\end{equation}
which is \eqref{nonmonformula} for $\tau=k\thinspace k-1\cdots1$ and $\eta=231$.
\hfill $\square$

\noindent \it Part (ii).\rm\ From \eqref{nonmonfinal231} and \eqref{monfinal231}, we have
\begin{equation}\label{EN}
\begin{aligned}
&E_n^{\text{av}(231)}N_n^{(k)}=(C_k-1)\frac{(n-k+2)C_{n-k+1}}{2C_n}+\\
&\frac{(n-k+3)C_{n-k+2}}{2C_n}-\frac{(n-k+2)C_{n-k+1}}{C_n}.
\end{aligned}
\end{equation}
Using the formula $C_{m+1}=\frac{2(2m+1)}{m+2}C_m$ to write $C_{n-k+2}$ in terms of $C_{n-k+1}$ in \eqref{EN}, and performing some algebra,
gives
$$
E_n^{\text{av}(231)}N_n^{(k)}=\frac{C_{n-k+1}}{2C_n}\Big((n-k+2)C_k+n-k\Big),
$$
which is \eqref{fullnonmon} for $\eta=231$.
\hfill $\square$

We now return to prove Proposition \ref{propcond}.

\noindent\it Proof of Proposition \ref{propcond}.\rm\
In order to make the proof of the proposition more transparent, we  first prove it for $k=6$ and three particular choices of $\tau$---216345, 621345 and 654321.
The  proofs of these particular cases will make  the general case much easier to follow. Note that the first two cases are particular cases of part (i) and the third case is a particular case of part (ii).

We begin with $\tau=216345$.
Note that for this case, the $i_k$ appearing in part (i) of the proposition is given by $i_6=3$.
For this case, \eqref{propcondgeneralnot1}-\eqref{propAl} become
\begin{equation}\label{cond216345}
N_n^{(6;216345)}|\{\sigma_j=n\}\stackrel{\text{dist}}{=}
N_{j-1}^{(6;216345)}+N_{n-j}^{(6;216345)}+1_{\{j=n-3\}}1_{A^{l}_{n-4}}1_{\sigma^{*,3}},
\end{equation}
where
\begin{equation}\label{sigma*3}
\sigma^{*,3}=123=\text{red}(\tau_4\tau_5\tau_6)\in S_{3}^{\text{av}(231)}
\end{equation}
and
\begin{equation}\label{Alparticular}
A^{l}_{n-4}=\{\sigma\in S_{n-4}^{\text{av}(231)}: \sigma_{n-5}=n-4,\sigma_{n-4}=n-5\},
\end{equation}
and where $N_{j-1}^{(6;216345)}$ is independent of $N_{n-j}^{(6;216345)}$, and the pair $N_{n-4}^{(6;216345)}, 1_{A^{l}_{n-4}}$ is independent of the pair $N_{3}^{(6;216345)},1_{\sigma^{*,3}}$.
We now verify \eqref{cond216345}-\eqref{Alparticular}.
The terms $N_{j-1}^{(6;216345)}$ and $N_{n-j}^{(6;216345)}$  on the right hand side of \eqref{cond216345} are clear; they count the number of clusters of length $k$ and   pattern 216345 from
$\sigma_1\cdots, \sigma_{j-1}$ and from $\sigma_{j+1}\cdots\sigma_n$ respectively. We now show that the term, $1_{\{j=n-3\}}1_{A^{l}_{n-4}}1_{\sigma^{*,3}}$ counts  clusters of length $6$ and pattern 216345 that involve the number $n=\sigma_j$.
Such a  cluster that includes the number $n=\sigma_j$ can only  be the cluster
\begin{equation}\label{cluster6-3}
\begin{aligned}
&\sigma_{j-2}=n-4,\sigma_{j-1}=n-5,\sigma_j=n,\sigma_{j+1}=n-3,\sigma_{j+2}=n-2,\sigma_{j+3}=n-1.
\end{aligned}
\end{equation}
Furthermore, since $\sigma\in S_n^{\text{av}(231)}$, \eqref{cluster6-3} can only occur if $j=n-3$. Indeed,
if $j>n-3$ or $j\le2$, then obviously \eqref{cluster6-3} cannot occur, while
 if $3\le j<n-3$, then it would follow from  \eqref{cluster6-3} that $\sigma_{j+4}\le n-6$. But then $\sigma_{j-1}\sigma_j\sigma_{j+4}$ would have the pattern 231.
  Finally, given that $\sigma_{n-3}=n$,  $\sigma\in S_n^{\text{av}(231)}$ satisfies \eqref{cluster6-3} with $j=n-3$ if and only if
$\sigma_1\cdots\sigma_{n-4}\in A^l_{n-4}$ and $\text{red}(\sigma_{n-2}\sigma_{n-1}\sigma_n)=\sigma^{*,3}$ where
$A^l_{n-4}$ is as in \eqref{Alparticular} and $\sigma^{*,3}$ is as in \eqref{sigma*3}.
The above-noted independence follows from the second paragraph of section \ref{thm1proof}.
In the sequel, we will refrain from commenting on the justification for independence, as it will always follow from the second paragraph of section \ref{thm1proof}.

Consider now the case $\tau=621345$.
Note that in this case, the $i_k$ appearing in part (i) of the proposition is given by $i_6=1$.
For this case, \eqref{propcondgeneralnot1}-\eqref{propAl} become
\begin{equation}\label{cond621345}
N_n^{(6;621345)}|\{\sigma_j=n\}\stackrel{\text{dist}}{=}
N_{j-1}^{(6;621345)}+N_{n-j}^{(6;621345)}+1_{\{j=n-5\}}1_{\sigma^{*,5}},
\end{equation}
where
\begin{equation}\label{sigma*5}
\sigma^{*,5}=21345=\text{red}(\tau_2\cdots\tau_6)
\end{equation}
and where $N_{j-1}^{(6;621345)}$ is independent of $N_{n-j}^{(6;621345)}$, and $N_{n-6}^{(6;621345)}$ is independent of the pair $N_{5}^{(6;621345)},1_{\sigma^{*,5}}$.
We now verify \eqref{cond621345}-\eqref{sigma*5}.
As before, the roles of $N_{j-1}^{(6;621345)}$ and $N_{n-j}^{(6;621345)}$ are clear. We show that the term $1_{\{j=n-5\}}1_{\sigma^{*,5}}$ counts those clusters of length $6$ and pattern 621345 that involve the number $n=\sigma_j$.
Such a cluster that includes the number $n=\sigma_j$ can only  be the cluster
\begin{equation}\label{cluster6-1}
\begin{aligned}
&\sigma_{j}=n,\sigma_{j+1}=n-4,\sigma_{j+2}=n-5,\sigma_{j+3}=n-3,\sigma_{j+4}=n-2,\sigma_{j+5}=n-1.
\end{aligned}
\end{equation}
Furthermore, since $\sigma\in S_n^{\text{av}(231)}$,
this can occur only if  $j=n-5$. Indeed, if $j>n-5$, then obviously \eqref{cluster6-1} cannot occur, while if $j<n-5$, then it would follow from \eqref{cluster6-1} that
$\sigma_{j+6}\le n-6$.
But then, for example, $\sigma_{j+1}\sigma_{j+3}\sigma_{j+6}$ would have the pattern 231.
Finally, given that $\sigma_{n-5}=n$, $\sigma\in S_n^{\text{av}(231)}$ satisfies \eqref{cluster6-1} with $j=n-5$ if and only if   $\text{red}(\sigma_{n-4}\cdots\sigma_n)=\sigma^{*,5}$,
where $\sigma^{*,5}$ is as in \eqref{sigma*5}.

The argument above for $\tau=621345$ works just as well for any other $\tau\in S_6$ with $\tau_1=6$, except for $\tau=654321$, which we now consider.
For this case, \eqref{propcondgeneral1mon}-\eqref{propAr} become
\begin{equation}\label{cond654321}
N_n^{(6;654321)}|\{\sigma_j=n\}\stackrel{\text{dist}}{=}
N_{j-1}^{(6;654321)}+N_{n-j}^{(6;654321)}+1_{\{j\le n-5\}}1_{A^{r}_{n-j}},
\end{equation}
where
\begin{equation}\label{Arparticular}
A^{r}_{n-j}=\{\sigma\in S_{n-j}^{\text{av}(231)}:\sigma_l=n-j+1-l,\ l=1,\cdots, 5\},
\end{equation}
and where $N_{j-1}^{(6;654321)}$ is independent of the pair $N_{n-j}^{(6;654321)},1_{A^{r}_{n-j}}$.
We now verify \eqref{cond654321}-\eqref{Arparticular}.
As before, the roles of $N_{j-1}^{(6;654321)}$ and $N_{n-j}^{(6;654321)}$ are clear.
We now show that the term  $1_{\{j\le n-5\}}1_{A^{r}_{n-j}}$ counts  clusters of length $6$ and pattern $654321$ that involve the number $n=\sigma_j$.
Such a cluster that includes the number $n=\sigma_j$ can only  be the cluster
\begin{equation}\label{cluster6-1again}
\begin{aligned}
&\sigma_{j}=n,\sigma_{j+1}=n-1,\sigma_{j+2}=n-2,\sigma_{j+3}=n-3,\sigma_{j+4}=n-4,\sigma_{j+5}=n-5.
\end{aligned}
\end{equation}
Of course, \eqref{cluster6-1again} cannot occur if $j>n-5$; this accounts for the term $1_{\{j\le n-5\}}$.
In the previous case, we also ruled out $j<n-5$ because that would lead to the existence of the pattern 231.
In the present case, because the pattern in \eqref{cluster6-1again} is  decreasing, the argument in the previous case no longer goes through.
Finally, for $j\le n-5$ and given that $\sigma_j=n$,  $\sigma\in S_n^{\text{av}(231)}$ satisfies \eqref{cluster6-1again} if and only if
$\text{red}(\sigma_{j+1}\cdots \sigma_n)\in A^{r}_{n-j}$, where $A^{r}_{n-j}$ is as in \eqref{Arparticular}.

With the above particular cases explained, we now turn to the proof of the general case.
We first assume that $i_k=\tau_k^{-1}\neq1$.
 We  will show that part (i) of the proposition holds.
The terms $N_{j-1}^{(k,\tau)}$ and $N_{n-j}^{(k;\tau)}$  on the right hand side of \eqref{propcondgeneralnot1} are clear; they count the number of clusters of length $k$ and   pattern $\tau$ from
$\sigma_1\cdots, \sigma_{j-1}$ and from $\sigma_{j+1}\cdots\sigma_n$. We now show that the term $1_{\{j=n-k+i_k\}}1_{A^{l}_{n-k+i_k-1}}1_{\sigma^{*,k-i_k}}$ counts  clusters of length $k$ and pattern $\tau$ that involve the number $n=\sigma_j$.
Of course, the only candidate for such a cluster is $\sigma_{j-i_k+1}\sigma_{j-i_k+2}\cdots\sigma_{j+k-i_k}$, and  this will indeed constitute such a cluster
if and only if
\begin{equation}\label{clustergeneralnot1}
\text{red}(\sigma_{j-i_k+1}\sigma_{j-i_k+2}\cdots\sigma_{j+k-i_k})=\tau.
\end{equation}

We now show that \eqref{clustergeneralnot1} can only occur if $j=n-k+i_k$. This will account for the term $1_{\{j=n-k+i_k\}}$.
Indeed, if $j>n-k+i_k$ or $j\le i_k-1$, then obviously \eqref{clustergeneralnot1} cannot occur.
Now consider  $i_k\le j<n-k+i_k$. Since $\{ \sigma_{j-i_k+1},\sigma_{j-i_k+2},\cdots\sigma_{j+k-i_k}\}=\{n-k+1,\cdots,n\}$, it would follow from
  \eqref{clustergeneralnot1} that
   $\sigma_{j+k-i_k+1}\le n-k$. But then $\sigma_{j-1}\sigma_j\sigma_{j+k-i_k+1}$ would have the pattern 231, which is forbidden.
   Thus, we conclude that there will be either one such cluster involving $n$ or no such cluster involving $n$, and the condition for the existence of such a cluster is that $j=n-k+i_k$ and
\begin{equation}\label{clustergeneralonejnot1}
\text{red}(\sigma_{n-k+1}\cdots\sigma_n)=\tau.
\end{equation}
  Finally, given $\sigma_{n-k+i_k}=n$,  $\sigma\in S_n^{\text{av}(231)}$ satisfies \eqref{clustergeneralonejnot1}  if and only if
$\sigma_1\cdots\sigma_{n-k+i_k-1}\in A^l_{n-k+i_k-1}$, where $A^l_{n-k+i_k-1}$ is as in \eqref{propAl}  and, if $i_k\neq k$, then also $\text{red}(\sigma_{n-k+i_k+1}\cdots\sigma_n)=\sigma^{*,k-i_k}$,
where $\sigma^{*,k-i_k}$ is as in \eqref{sigma*first}.

We now consider the case that $i_k=\tau_k^{-1}=1$.
Here we need to consider two subcases---the case that $\tau\neq kk-1\cdots 1$, and the case that $\tau=kk-1\cdots 1$.
We first consider the subcase that  $\tau\neq kk-1\cdots 1$.
We will show that part (i) of the proposition holds.
Again, the roles of $N_{j-1}^{(k;\tau)}$ and $N_{n-j}^{(k;\tau)}$ are clear.
We now show that the term $1_{\{j=n-k+1\}}1_{\sigma^{*,k-1}}$ counts  clusters of length $k$ and pattern $\tau$ that involve the number $n=\sigma_j$.
Of course, the only candidate for such a cluster is $\sigma_{j}\sigma_{j+1}\cdots\sigma_{j+k-1}$, and  this will indeed constitute such a cluster
if and only if
\begin{equation}\label{clustergeneral1notmon}
\text{red}(\sigma_{j}\sigma_{j+1}\cdots\sigma_{j+k-1})=\tau.
\end{equation}
We now show that \eqref{clustergeneral1notmon} can only occur if $j=n-k+1$. This will account for the term $1_{\{j=n-k+1\}}$.
Indeed, if $j>n-k+1$, then obviously \eqref{clustergeneral1notmon} cannot occur.
Now consider  $j<n-k+1$. Since $\{ \sigma_{j},\sigma_{j+1},\cdots\sigma_{j+k-1}\}=\{n-k+1,\cdots,n\}$, it would follow from
  \eqref{clustergeneral1notmon} that
   $\sigma_{j+k}\le n-k$. Also, since by assumption, $\tau_2\cdots \tau_k\neq k-1\cdots1$,
   it follows from \eqref{clustergeneral1notmon} that
   there exist indices $l_1<l_2$ from the set $\{j+1,\cdots,j+k-1\}$ such that
   $\sigma_{l_1}<\sigma_{l_2}$. But then $\sigma_{l_1}\sigma_{l_2}\sigma_{j+k}$ would have the pattern 231, which is forbidden.
   Thus, we conclude that there will be either one such cluster involving $n$ or no such cluster involving $n$, and the condition for the existence of such a cluster is $j=n-k+1$ and
\begin{equation}\label{clustergeneralonej1notmon}
\text{red}(\sigma_{n-k+1}\cdots\sigma_n)=\tau.
\end{equation}
 Finally, given  $\sigma_{n-k+1}=n$, $\sigma\in S_n^{\text{av}(231)}$ satisfies \eqref{clustergeneralonej1notmon}  if and only if
  $\text{red}(\sigma_{n-k+2}\cdots\sigma_n)=\sigma^{*,k-1}$, where $\sigma^{*,k-1}$ is as in \eqref{sigma*first} with $i_k=1$.

We now turn to the subcase $\tau=k\thinspace k-1\cdots 1$ of the case $i_k=\tau_k^{-1}=1$. We will show that
part (ii) of the proposition holds.
Again, the roles of $N_{j-1}^{(k;\tau)}$ and $N_{n-j}^{(k;\tau)}$ are clear.
We now show that the term  $1_{\{j\le n-k+1\}}1_{A^r_{n-j}}$ counts  clusters of length $k$ and pattern $\tau$ that involve the number $n=\sigma_j$.
Of course, the only candidate for such a cluster is $\sigma_{j}\sigma_{j+1}\cdots\sigma_{j+k-1}$, and  this will indeed constitute such a cluster
if and only if
\begin{equation}\label{clustergeneral1mon}
\text{red}(\sigma_{j}\sigma_{j+1}\cdots\sigma_{j+k-1})=\tau=k\cdots1.
\end{equation}
Of course, \eqref{clustergeneral1mon} cannot occur if $j>n-k+1$; this accounts for the term $1_{\{j\le n-k+1\}}$.
Finally, given that $\sigma_j=n$, it follows that for $1\le j\le n-k+1$,   $\sigma\in S_n^{\text{av}(231)}$ satisfies \eqref{clustergeneral1mon} if and only if
$\text{red}(\sigma_{j+1}\cdots \sigma_n)\in A^r_{n-j}$, where $A_{n-j}^r$ is as in \eqref{propAr}.
\hfill$\square$

\section{Proof of Proposition \ref{nonmonpropvar}}\label{varproof}
As noted before the statement of Theorem \ref{nonmonthm}, it suffices to consider the case $\eta=231$. Let $2\le k\le n$.
We will prove the proposition for the case that $\tau\in S_k^{\text{av}(231)}$ satisfies $\tau\neq k\cdots1$.
The proof uses part (i) of Proposition \ref{propcond}.
The case $\tau=k\cdots1$ is treated similarly, using part (ii) of that proposition.
From part (i) of  Proposition \ref{propcond}, it follows that for $2\le k\le n$,
\begin{equation}\label{squaredequ}
\begin{aligned}
&(N_n^{(k;\tau)})^2|\{\sigma_j=n\}\stackrel{\text{dist}}{=}(N_{j-1}^{(k;\tau)})^2+(N_{n-j}^{(k;\tau)})^2+2N_{j-1}^{(k;\tau)}N_{n-j}^{(k;\tau)},\ j\neq n-k+i_k;\\
&(N_n^{(k;\tau)})^2|\{\sigma_j=n\}\stackrel{\text{dist}}{=}(N_{j-1}^{(k;\tau)})^2+(N_{n-j}^{(k;\tau)})^2+2N_{j-1}^{(k;\tau)}N_{n-j}^{(k;\tau)}+\\
&2\big(N_{j-1}^{(k;\tau)}+N_{n-j}^{(k;\tau)}\big)1_{A^{l}_{n-k+i_k-1}}1_{\sigma^{*,k-i_k}}+1_{A^{l}_{n-k+i_k-1}}1_{\sigma^{*,k-i_k}},\ j=n-k+i_k,
\end{aligned}
\end{equation}
where the notation in \eqref{squaredequ} is as in Proposition \ref{propcond}.
Let $r_n^\tau=E_n^{\text{av}(231)}(N_n^{(k;\tau)})^2$ and recall from section \ref{thm1proof}
that $s_n^\tau=E_n^{\text{av}(231)}N_n^{(k;\tau)}$.
From \eqref{squaredequ} and \eqref{nprob}, along with the independence of certain random variables as indicated  in Proposition \ref{propcond}, it follows
that
\begin{equation*}
\begin{aligned}
&r_n^\tau=2\sum_{j=1}^n\frac{C_{j-1}C_{n-j}}{C_n}r^\tau_{j-1}+2\sum_{j=1}^n\frac{C_{j-1}C_{n-j}}{C_n}s_{j-1}^\tau s_{n-j}^\tau+\\
&\frac{C_{n-k+i_k-1}C_{k-i_k}}{C_n}\big(2E_{n-k+i_k-1}^{\text{av}(231)}N_{n-k+i_k-1}^{(k;\tau)}1_{A^{l}_{n-k+i_k-1}}\big)P_{k-i_k}^{\text{av}(231)}(\sigma^{*,k-i_k})+\\
&\frac{C_{n-k+i_k-1}C_{k-i_k}}{C_n}P_{n-k+i_k-1}^{\text{av}(231)}(A_{n-k+i_k-1}^l)
P_{k-i_k}^{\text{av}(231)}(\sigma^{*,k-i_k}),
\end{aligned}
\end{equation*}
 where we have used the fact that $N_n^{k;\tau}=0$, for $n<k$.
 Using the fact that $P_{k-i_k}^{\text{av}(231)}(\sigma^{*,k-i_k})=
 \frac1{C_{k-i_k}}$ along with
 \eqref{probAlAr}, we can rewrite the above equation as
\begin{equation}\label{rn}
\begin{aligned}
&r_n^\tau=2\sum_{j=1}^n\frac{C_{j-1}C_{n-j}}{C_n}r^\tau_{j-1}+2\sum_{j=1}^n\frac{C_{j-1}C_{n-j}}{C_n}s_{j-1}^\tau s_{n-j}^\tau+\\
&2\frac{C_{n-k+i_k-1}}{C_n}E_{n-k+i_k-1}^{\text{av}(231)}N_{n-k+i_k-1}^{(k;\tau)}1_{A^{l}_{n-k+i_k-1}}+
\frac{C_{n-k}}{C_n}.
\end{aligned}
\end{equation}

Define
$$
W(t)=\sum_{n=k}^\infty C_{n-k+i_k-1}\big(E_{n-k+i_k-1}^{\text{av}(231)}N_{n-k+i_k-1}^{(k;\tau)}1_{A^{l}_{n-k+i_k-1}}\big)t^n.
$$
Note that
\begin{equation}\label{west}
E_{n-k+i_k-1}^{\text{av}(231)}N_{n-k+i_k-1}^{(k;\tau)}1_{A^{l}_{n-k+i_k-1}}\le E_{n-k+i_k-1}^{\text{av}(231)}N_{n-k+i_k-1}^{(k;\tau)}=s_{n-k+i_k-1}^\tau.
\end{equation}
Recalling the definition of $G^{(k;\tau)}(t)$ in \eqref{G}, we note for later use that  for each $n$,
the coefficient of $t^n$ in the power series defining $W(t)$ is less than or equal to the coefficient of $t^n$ in the power series for
$G^{(k;\tau)}(t)$.
This  follows
from \eqref{west}.
Define
$$
R^{(k;\tau)}(t)=\sum_{n=k}^\infty C_nr_n^\tau t^n.
$$
Multiplying \eqref{rn} by $C_nt^n$ and summing over $n$ from $k$ to $\infty$, and recalling the definition of $G^{(k;\tau)}(t)$, we obtain
\begin{equation*}
R^{(k;\tau)}(t)=2tC(t)R^{(k;\tau)}(t)+2t(G^{(k;\tau)}(t))^2+2t^{k-i_k+1}W(t)+t^kC(t),
\end{equation*}
from which it follows that
\begin{equation}\label{Requ}
R^{(k;\tau)}(t)=\frac{2t(G^{(k;\tau)}(t))^2+2t^{k-i_k+1}W(t)+t^kC(t)}{1-2tC(t)}.
\end{equation}

Using \eqref{Catalangenexpl} and  \eqref{formulaG}, we obtain after some algebra,
\begin{equation}\label{leadterm}
\frac{2t(G^{(k;\tau)}(t))^2}{1-2tC(t)}=(t^{2k-1}-2t^{2k})(1-4t)^{-\frac32}-t^{2k-1}(1-4t)^{-1}.
\end{equation}
Also,
\begin{equation}\label{anotherterm}
\frac{t^kC(t)}{1-2tC(t)}=\frac12t^{k-1}(1-4t)^{-\frac12}-\frac12t^{k-1}.
\end{equation}
From \eqref{nthderiv-12}, we have
\begin{equation}\label{series-onehalf}
(1-4t)^{-\frac12}=\sum_{n=0}^\infty (n+1)C_nt^n.
\end{equation}
Differentiating \eqref{series-onehalf}, we obtain
\begin{equation}\label{series-threehalves}
(1-4t)^{-\frac32}=\frac12\sum_{n=0}^\infty(n+1)(n+2)C_{n+1}t^n.
\end{equation}
Since
$\frac{t^{k-i_k+1}W(t)}{1-2tC(t)}=t^{k-i_k+1}W(t)(1-4t)^{-\frac12}$, and since the coefficients of the power series for $W(t)$ are all nonnegative and are dominated by those of the power series
for $G^{(k;\tau)}(t)$,
it follows  from \eqref{series-onehalf} that
\begin{equation}\label{verbalequ}
\begin{aligned}
&\text{the coefficients of the power series for}\ \frac{t^{k-i_k+1}W(t)}{1-2tC(t)}\ \text{are dominated by those}\\
&\text{for the power series}\
\frac{t^{k-i_k+1}G^{(k;\tau)}(t)}{1-2tC(t)}.
\end{aligned}
\end{equation}
By \eqref{Catalangenexpl} and \eqref{formulaG},
\begin{equation}\label{upperbdseries}
\frac{t^{k-i_k+1}G^{(k;\tau)}(t)}{1-2tC(t)}=\frac12t^{2k-i_k}\big((1-4t)^{-1}-(1-4t)^{-\frac12}\big).
\end{equation}

Since $C_n\sim\frac{4^n}{\sqrt\pi n^\frac32}$, it follows from \eqref{Requ}-\eqref{upperbdseries} that
the leading order term as $n\to\infty$ of the coefficient $C_nr_n^\tau$ of $t^n$ in the power series for $R^{(k;\tau)}(t)$  comes from the coefficient of $t^n$ in the power series for
$(t^{2k-1}-2t^{2k})(1-4t)^{-\frac32}$ on the right hand side of \eqref{leadterm}. Using this with \eqref{series-threehalves}, we obtain
\begin{equation}\label{cnrn}
C_nr_n^\tau\sim \frac12(n-2k+2)(n-2k+3)C_{n-2k+2}-(n-2k+1)(n-2k+2)C_{n-2k+1}.
\end{equation}
Since $\lim_{n\to\infty}\frac{C_{n-1}}{C_n}=\frac14$, it follows from \eqref{cnrn} that
\begin{equation}\label{rnasymp}
E_n^{\text{av}(231)}(N_n^{(k;\tau)})^2=r_n^\tau\sim n^2\big(\frac124^{2-2k}-4^{1-2k}\big)=\frac{n^2}{4^{2k-1}}.
\end{equation}
From \eqref{asymplim}, we have
\begin{equation}\label{firstmoment}
E_n^{\text{av}(231)}N_n^{(k;\tau)}\sim \frac n{2\cdot4^{k-1}} .
\end{equation}
Thus, it follows from \eqref{rnasymp} and \eqref{firstmoment} that
$$
\text{Var}(N_n^{(k;\tau)})=o(n^2).
$$
\hfill $\square$
\section{Proof of Theorem \ref{monthm}}\label{thm2proof}
As noted in the introduction, it suffices to consider the case $\eta=321$.
The key step to proving the theorem is the following result, whose rather technical proof will be postponed until after the proof of Theorem \ref{monthm}.
Recall from the introduction that  $A_{n}^{k;l}\subset S_n$ denotes the event that
 the set of  $k$ consecutive numbers $\{l, l+1,\cdots, l+k-1\}$ appears
in a  set  of consecutive positions. For $\tau\in S_k$, let $A_{n}^{k,\tau;l}\subset S_n$ denote the event that the set of $k$ consecutive numbers $\{l, l+1,\cdots, l+k-1\}$ appears
in a  set  of consecutive positions according to the pattern $\tau$.
\begin{proposition}\label{123321}
 For $n\ge k\ge2$ and $1\le l \le n-k+1$,
 \begin{equation}\label{exact321}
 P_n^{\text{av}(321)}(A_n^{k,\tau;l})=\begin{cases}\frac{C_{l-1}C_{n-k-l+1}}{C_n},\ \text{if}\ 1\cdots k\neq \tau \in S_k^{\text{av}(321)};\\ \frac{C_{n-k+1}}{C_n},\ \text{if}\ \tau=1\cdots k.
 \end{cases}
 \end{equation}
\end{proposition}
The proof of Theorem \ref{monthm} is almost immediate from Proposition \ref{123321}.

\noindent \it Proof of Theorem \ref{monthm}.\rm\
Since $N_{n}^{(k;\tau)}=\sum_{l=1}^{n-k+1}1_{A_{n}^{k,\tau;l}}$,  Proposition \ref{123321} yields
\begin{equation*}
\begin{aligned}
&E_n^{\text{av}(321)}N_n^{(k,\tau)}=\sum_{l=1}^{n-k+1} P_n^{\text{av}(321)}(A_n^{k,\tau;l})=\begin{cases}\frac{C_{n-k+1}}{C_n},\ \text{if}\ 1\cdots k\neq\tau\in S_k^{\text{av}(321)};\\
\frac{(n-k+1)C_{n-k+1}}{C_n},\ \text{if}\ \tau=1\cdots k,\end{cases}
\end{aligned}
\end{equation*}
where the latter equality in the case $\tau\neq 1\cdots k$ follows from \eqref{catrecur}.
This gives \eqref{monformula} in the case $\eta=321$.
Since there are $C_k-1$ permutations $1\cdots k\neq\tau\in S_k^{\text{av}(321)}$,
\eqref{fullmon} follows from \eqref{monformula}.
\medskip

\noindent \it Proof of Proposition \ref{123321}.\rm\
Fix $n,k,l$ as in the statement of the proposition. Fix $\tau\in S_k^{\text{av}(321)}$.
For $1\le a\le n-k+1$, define
$$
A_n^{k,\tau;l,a}=\big\{\sigma\in A_n^{k,\tau;l}:\{l,l+1,\cdots, l+k-1\}=\{\sigma_a,\sigma_{a+1},\cdots\sigma_{a+k-1}\big\}.
$$
Then the sets $\{A_n^{k,\tau;l,a}\}_{a=1}^{n-k+1}$ are disjoint and
$A_n^{k,\tau;l}=\cup_{a=1}^{n-k+1}A_n^{k,\tau;l;a}$.

If $\nu=\{\nu_i\}_{i=1}^{|B|}$ is a permutation of a finite set  $B\subset\mathbb{N}$, let
$\nu^{B^{-1}}$ denote the permutation it naturally induces on $S_{|B|}$; that is, $\nu^{B^{-1}}=\text{red}(\nu)$, where $\text{red}(\cdot)$ was defined at the beginning of section \ref{thm1proof}.
Conversely, if  $\nu$ is a permutation of $S_{|B|}$, let   $\nu^{B}$ denote
the permutation it naturally induces on $B$.

Until further notice,  consider $a$ fixed.
Let $\sigma\in A_n^{k,\tau;l;a}\cap S_n^{\text{av}(321)}$.
We  describe a procedure    to contract $\sigma$ to a permutation in $S_{n-k+1}^{\text{av}(321)}$.
Define the permutation $\overline\sigma=\overline\sigma(\sigma)$ of the set
$B=\{1,\cdots, l,l+k,\cdots, n\}$  by
$$
\overline\sigma_i=\begin{cases}\sigma_i,\ 1\le i\le a-1;\\ l,\ i=a;\\ \sigma_{i+k-1}, \ i=a+1,\cdots, n-k+1,\end{cases}
$$
and define
$$
\nu=\nu(\sigma)=\overline\sigma^{B^{-1}}.
$$
It  follows from the construction  that
\begin{equation}\label{nu}
\nu\in S_{n-k+1}^{\text{av}(321)} \ \text{and}\ \nu_a=l.
\end{equation}

We concretize the above construction with an example.  Let $n=9, k=3, l=a=4$.
Consider $\sigma=213546897\in A_9^{3,213;4;4}\cap S_9^{\text{av}(321)}$. The set $B$ is given by $B=\{1,2,3,4,7,8,9\}$ and $\bar\sigma=2134897$---the cluster $546$  in $\sigma$ has been
contracted to $4$ in $\bar\sigma$. Finally, $\nu=\nu(\sigma)=\bar\sigma^{B^{-1}}=2134675$ satisfies $\nu\in S_7^{\text{av}(321)}$ and $\nu_4=4$.

Obviously the map taking $\sigma\in A_n^{k,\tau;l;a}\cap S_n^{\text{av}(321)}$ to $\nu(\sigma)$ is not injective.
However,
\begin{equation}\label{partialinject}
\begin{aligned}
&\nu(\sigma)\neq\nu(\sigma'), \ \text{if}\ \sigma,\sigma'\in  A_n^{k,\tau;l;a}\cap S_n^{\text{av}(321)}\ \text{are distinct and satisfy}\\
&\sigma_{a+i}=\sigma'_{a+i}, \ i=0,\cdots, k-1.
\end{aligned}
\end{equation}

Conversely, let $\nu$ satisfy \eqref{nu}. We describe a procedure to extend
$\nu$  to a permutation in $A_n^{k,\tau;l;a}$, which may or may not belong to $S_n^{\text{av}(321)}$.
Let $B=\{1,\cdots, l,l+k,\cdots, n\}$ as above.
Define
 $\sigma^{\tau}=\sigma^{\tau}(\nu)\in S_n$ by
\begin{equation*}\label{nusigma}
\sigma_i^{\tau}=\begin{cases}\nu^B_i,\ i=1,\cdots, a-1;\\ l-1+\tau_{i-a+1},\ i=a,\cdots, a+k-1;\\
\nu^B_{i-k+1},\ i=a+k,\cdots, n.
\end{cases}
\end{equation*}
It follows from the construction that
\begin{equation}\label{sigmatau}
\sigma^\tau(\nu)\in A_n^{k,\tau;l;a}.
\end{equation}
Also, of course, the map taking $\nu$ satisfying \eqref{nu} to $\sigma^\tau(\nu)$ is injective.

As an example of the above construction, again with $n=9, k=3, l=a=4$, consider $\nu=2134675$. Then $\nu$ satisfies  \eqref{nu}.
We have $B=\{1,2,3,4,7,8,9\}$. Choose, for example, $\tau=213\in S_3$. Then $\nu^B=2134897$ and $\sigma^\tau=\sigma^\tau(\nu)=213546897\in A_9^{3,213;4,4}$---the $4$ in
$\nu^B$ has been expanded to the cluster $546$ in $\sigma^\tau$.

We now investigate when in fact $\sigma^\tau(\nu)\in S_n^{\text{av}(321)}$, or equivalently in light of \eqref{sigmatau}, when $\sigma^\tau(\nu)\in S_n^{\text{av}(321)}\cap A_n^{k,\tau;l;a}$.
If $\tau=12\cdots k$, then it is clear that $\sigma^\tau(\nu)\in S_n^{\text{av}(321)}$, for all
$\nu$ satisfying \eqref{nu}.
Thus, since the map taking $\sigma\in A_n^{k,\tau;l;a}\cap S_n^{\text{av}(321)}$ to $\nu(\sigma)$ satisfies \eqref{partialinject}, and the map taking
$\nu$ satisfying \eqref{nu} to $\sigma^\tau(\nu)$ is injective, it follows that
\begin{equation}\label{countfor1tok}
|S_n^{\text{av}(321)}\cap A_n^{k,\tau;l;a}|=|\{\nu\in S_{n-k+1}^{\text{av}(321)}:\nu_a=l\}|,\ \tau=1\cdots k.
\end{equation}

Now consider any of the other $\tau\in S_k^{\text{av}(321)}$.
Since $\tau$ has a decreasing subsequence of length 2, in order
to have $\sigma^\tau(\nu)\in S_n^{\text{av}(321)}$, all of the numbers $\{1,\cdots, l-1\}$ must appear among the first $a-1$ positions of $\nu$, and all the numbers $\{l+1,\cdots n-k+1\}$ must appear
 among the last  $n-a-k+1$ positions of $\nu$. This is possible only if $a=l$.
 If indeed $a=l$, then $\sigma^\tau(\nu)\in S_n^{\text{av}(321)}$ if
 and only if  the first $l-1$ positions of $\nu$ are filled in a 321-avoiding way by the numbers
$\{1,\cdots, l-1\}$  and the last  $n-l-k+1$ positions of $\nu$ are filled in a 321-avoiding
way by the numbers $\{l+1,\cdots n-k+1\}$. (The one remaining position, position $a$, is by assumption filled by the number $l$.)
Thus, again
because the map taking $\sigma\in A_n^{k,\tau;l;a}\cap S_n^{\text{av}(321)}$ to $\nu(\sigma)$ satisfies \eqref{partialinject}, and the map taking
$\nu$ satisfying \eqref{nu} to $\sigma^\tau(\sigma)$ is injective, it follows that
\begin{equation}\label{countfornot1tok}
|S_n^{\text{av}(321)}\cap A_n^{k,\tau;l;a}|=\begin{cases} C_{l-1}C_{n-l-k+1},\ a=l;\\ 0,\ a\neq l.\end{cases}\ \text{for}\ 1\cdots k\neq\tau\in S_k^{\text{av}(321)}.
\end{equation}

Summing \eqref{countfor1tok} and \eqref{countfornot1tok} over $a\in\{1,\cdots, n-k+1\}$, we obtain
\begin{equation}\label{almostfinalstep}
|S_n^{\text{av}(321)}\cap A_n^{k,\tau;l}|=\begin{cases} C_{n-k+1},\ \tau=1\cdots k;\\ C_{l-1}C_{n-l-k+1},\ 1\cdots k\neq\tau\in S_k^{\text{av}(321)}.\end{cases}
\end{equation}
From \eqref{almostfinalstep}, we conclude that \eqref{exact321} holds.
\hfill $\square$

\section{Proof of Theorem \ref{multipatternthm}}\label{thm3proof}
With just one change, we follow the construction appearing in the  proof of Proposition \ref{123321} above, from the beginning of the proof  up until but not including the paragraph containing
\eqref{countfor1tok}. The one change is that wherever $S_m^{\text{av}(321)}$ appears, for some $m\in\mathbb{N}$, it needs to be replaced by
$S_m^{\text{av}(\eta_1,\cdots,\eta_r)}$. (Therefore, the two examples appearing in the construction also needed to be amended.) Thus, in the sequel,  whenever we refer to equations appearing in the above noted construction,
any appearance of   $S_m^{\text{av}(321)}$ in such an equation
 must be changed to $S_m^{\text{av}(\eta_1,\cdots,\eta_r)}$. As in the proof of Proposition \ref{123321},
we now investigate when in fact $\sigma^\tau(\nu)\in S_n^{\text{av}(\eta_1,\cdots,\eta_r)}$, or equivalently
in light of \eqref{sigmatau}, when
$\sigma^\tau(\nu)\in S_n^{\text{av}(\eta_1,\cdots,\eta_r)}\cap A_n^{k,\tau;l;a}$.
In fact, this holds for all $\tau\in S_k^{\text{av}(\eta_1,\cdots,\eta_r)}$. Indeed, since by \eqref{nu},
$\nu\in S_{n-k+1}^{\text{av}(\eta_1,\cdots,\eta_r)}$, and since $\tau\in S_k^{\text{av}(\eta_1,\cdots,\eta_r)}$,  it follows from the definition of
$\sigma^\tau(\nu)$ that if $\sigma^\tau(\nu)\not\in S_n^{\text{av}(\eta_1,\cdots,\eta_r)}$, then  for some $i\in [r]$ and some $2\le j_0\le k-1$,
$\sigma^\tau(\nu)$ contains the pattern $\eta_i$ and  exactly $j_0$ of the numbers in
$\{\sigma^\tau_a(\nu),\cdots, \sigma_{a+k-1}^\tau(\nu)\}=\{l,\cdots, l+k-1\}$ are used in the construction of the pattern $\eta_i$.
But then it would follow that $\eta_i$ has a nontrivial block of length $j_0$, which contradicts the assumption that $\eta_i$ is simple.

Since for all $\tau\in S_k^{\text{av}(\eta_1,\cdots,\eta_r)}$, we have $\sigma^\tau(\nu)\in S_n^{\text{av}(\eta_1,\cdots,\eta_r)}\cap A_n^{k,\tau;l;a}$,
and since the map taking $\sigma\in A_n^{k,\tau;l;a}\cap S_n^{\text{av}(\eta_1,\cdots,\eta_r)}$ to $\nu(\sigma)$ satisfies \eqref{partialinject}, and the map taking
$\nu$ satisfying \eqref{nu} to $\sigma^\tau(\nu)$ is injective, it follows that
\begin{equation}\label{alltauequal}
|S_n^{\text{av}(\eta_1,\cdots,\eta_r)}\cap A_n^{k,\tau;l;a}|=|\{\nu\in S_{n-k+1}^{\text{av}(\eta_1,\cdots,\eta_r)}:\nu_a=l\}|.
\end{equation}
Summing \eqref{alltauequal} over   $a\in\{1,\cdots, n-k+1\}$, we obtain
\begin{equation}\label{afinal}
|S_n^{\text{av}(\eta_1,\cdots,\eta_r)}\cap A_n^{k,\tau;l}|=|S_{n-k+1}^{\text{av}(\eta_1,\cdots,\eta_r)}|.
\end{equation}
From \eqref{afinal}, we obtain
\begin{equation}\label{afinalagain}
P_n^{\text{av}(\eta_1,\cdots,\eta_r)}(A_n^{k,\tau;l})=\frac{|S_{n-k+1}^{\text{av}(\eta_1,\cdots,\eta_r)}|}
{|S_n^{\text{av}(\eta_1,\cdots,\eta_r)}|}.
\end{equation}
(We note that this construction leading to \eqref{afinalagain} is similar  to a construction in  \cite{AA}.)
Since
$$
E_n^{\text{av}(\eta_1,\cdots,\eta_r)}N^{(k;\tau)}=\sum_{l=1}^{n-k+1}P_n^{\text{av}(\eta_1,\cdots,\eta_r)}(A_n^{k,\tau;l}),
$$
 \eqref{multipattern} follows from \eqref{afinalagain}, and then \eqref{multipatternall} follows by summing \eqref{multipattern} over
 $\tau\in S_k^{\text{av}(\eta_1,\cdots,\eta_r)}$.
 \hfill $\square$

\end{document}